\documentclass[]{article}
\usepackage{lineno,hyperref}

\usepackage{amsmath, amssymb, amsfonts, amsthm, pgf, tikz, enumerate}
\usepackage{bm}
\usetikzlibrary{calc,patterns,
                 decorations.pathmorphing,
                 decorations.markings}
\usepackage[linesnumbered,ruled,vlined,boxed,commentsnumbered]{algorithm2e}
\usepackage{xcolor}
\definecolor{redblue}{rgb}{0.9,0,0.4}
\definecolor{bluered}{rgb}{0.2,0,0.9}
\definecolor{light-red}{rgb}{0.9,0.15,0.15}
\definecolor{dark-red}{rgb}{0.4,0.15,0.15}
\definecolor{redblue}{rgb}{0.9,0,0.4}
\definecolor{bluered}{rgb}{0.2,0,0.9}
\definecolor{light-red}{rgb}{0.9,0.15,0.15}
\definecolor{dark-red}{rgb}{0.4,0.15,0.15}
\definecolor{dark-blue}{rgb}{0.15,0.15,0.4}
\definecolor{medium-blue}{rgb}{0,0,0.5}

\theoremstyle{plain}
\newtheorem{theorem}{Theorem}[section]

\newtheorem{lemma}[theorem]{Lemma}
\newtheorem{corollary}[theorem]{Corollary}

\theoremstyle{definition}
\newtheorem{definition}[theorem]{Definition}

\newtheorem{example}[theorem]{Example}

\DeclareMathOperator{\at}{\,{\rule[-3mm]{.1mm}{8mm}}\,}
\DeclareMathOperator{\jac}{Jac}
\DeclareMathOperator{\diag}{diag}

\begin{document}

\title{Inverse Spectral Problems for Linked Vibrating Systems and Structured Matrix Polynomials}
\author{Keivan Hassani Monfared and Peter Lancaster\\ k1monfared@gmail.com and lancaste@ucalgary.ca \\ Department of Mathematics and Statistics, University of Calgary}
\maketitle

\begin{abstract}
	We show that for a given set $\Lambda$ of $nk$ distinct real numbers $\lambda_1, \lambda_2, \ldots, \lambda_{nk}$ and $k$ graphs on $n$ nodes, $G_0, G_1,\ldots,G_{k-1}$, there are real symmetric $n\times n$ matrices $A_s$, $s=0,1,\ldots, k$, such that the matrix polynomial $A(z) := A_k z^k + \cdots + A_1 z + A_0$ has $\Lambda$ as its spectrum, the graph of $A_s$ is $G_s$ for $s=0,1,\ldots,k-1$, and $A_k$ is an arbitrary positive definite diagonal matrix. When $k=2$, this solves a physically significant inverse eigenvalue problem for linked vibrating systems (see Corollary 5.3).
\end{abstract}

	

\section{Introduction}
	Inverse eigenvalue problems are of interest in both theory and applications. See, for example, the book of Gladwell \cite{g04} for applications in mechanics, the review article by Chu and Golub \cite{cg02} for linear problems, the monograph by Chu and Golub \cite{cg05} for general theory, algorithms and applications, and many references collected from various disciplines. In particular, the  \textit{Quadratic Inverse Eigenvalue Problems} (QIEP) are important and challenging because the general techniques for solving {\em linear} inverse eigenvalue problems cannot be applied directly. We emphasize that the structure, or linkage, imposed here is a feature of the physical systems illustrated in Section 2, and ``linked'' systems of this kind are our main concern.

Although the QIEP is important, the theory is presented here in the context of higher degree inverse spectral problems, and this introduction serves to set the scene and provide motivation for the more general theory developed in the main body of the paper -- starting with Section 3. The techniques used here generate systems with entirely real spectrum and perturbations which preserve this property.  Although, the method could be generalized to admit non-real conjugate pairs in the spectrum and the associated oscillatory behaviour. For example, the linear eigenvalue problem with conjugate pairs is solved in \cite{h17}.

QIEPs appear repeatedly in various scientific areas including structural mechanics, acoustic systems, electrical oscillations, fluid mechanics, signal processing, and finite element discretisation of partial differential equations. In most applications properties of the underlying physical system determine the parameters (matrix coefficients), while the behaviour of the system can be interpreted in terms of associated eigenvalues and eigenvectors. See Sections 5.3 and 5.4 of \cite{cg05}, where symmetric QIEPs are discussed. 

In this article it will be convenient to distinguish an eigenvalue of a matrix from a zero of the determinant of a matrix-valued function, which we call a {\em proper value}. (Thus, an eigenvalue of matrix $A$ is a proper value of $Iz-A$.) Given a quadratic matrix polynomial 
\begin{equation}\label{eq1} 
	L(z) = M z^2 + D z + K,\hspace*{1cm} M,D,K \in \mathbb{R}^{n\times n},
\end{equation} 
the direct problem is to find scalars $z_0$ and nonzero vectors\footnote{It is our convention to write members of $\mathbb{R}^n$ as {\bf column} vectors unless stated otherwise, and to denote them with bold lower case letters.} $\bm x \in C^n$ satisfying $L(z_0)\bm x =\bm 0$. The scalars $z_0$ and the vectors $\bm x$ are, respectively, {\em proper values} and {\em proper vectors} of the quadratic matrix polynomial $L(z)$. 

	Many applications, mathematical properties, and a variety of numerical techniques for the direct quadratic problem are surveyed in \cite{tm01}. On the other hand, the ``pole assignment problem'' can be examined in the context of a quadratic inverse eigenvalue problem \cite{nk01,der97,cd96,c02}, and a general technique for constructing families of quadratic matrix polynomials with prescribed semisimple eigenstructure was proposed in \cite{l07}. In \cite{bcs07} the authors address the problem when a partial list of eigenvalues and eigenvectors is given, and they provide a quadratically convergent Newton-type method. Cai et al.~in \cite{cklx09} and Yuan et al.~in \cite{yd11} deal with problems in which complete lists of eigenvalues and eigenpairs (and no definiteness constraints are imposed on $M,\;D,\;K$). In \cite{re96} and \cite{b08} the symmetric tridiagonal case with a partial list of eigenvalues and eigenvectors is discussed. 
        
	 A symmetric inverse quadratic proper value problem calls for the construction of a family of real symmetric quadratic matrix polynomials (possibly with some definiteness restrictions on the coefficients) consistent with prescribed spectral data \cite{lz14}. 

	An inverse proper value problem may be ill-posed \cite{cg05}, and this is particularly so for inverse quadratic proper value problems (IQPVP) arising from applications. This is because structure imposed on an IQPVP depends inherently on the connectivity of the underlying physical system. In particular, it is frequently necessary that, in the inverse problem, the reconstructed system (and hence the matrix polynomial) satisfies a \emph{connectivity} structure (see Examples \ref{ex:linearsystem} and \ref{ex:firstgraph}). In particular, the quadratic inverse problem for physical systems with a {\em serially linked structure} is studied in \cite{cdy07}, and there are numerous other studies on generally linked structures (see \cite{dlc09,ldc10-1,ldc10-2}, for example). 

In order to be precise about ``linked structure'' we need the following definitions:\\
A {\em (simple) graph} $G=(V,E)$ consists of two sets $V$ and $E$, where $V$, the set of {\em vertices} $v_i$ is, in our context, a finite subset of positive integers, e.g. $V = \{1,2,\ldots,n \}$, and $E$ is a set of pairs of vertices $\{v_i,v_j\}$ (with $v_i \neq v_j$) which are called the {\em edges} of $G$. (In the sense of \cite{hs13}, the graphs are ``loopless''.) 

If $\{v_i,v_j\} \in E$ we say $v_i$ and $v_j$ are \emph{adjacent} (See \cite{bm08}). Clearly, the number of edges in $G$ cannot exceed $\frac{n(n-1)}{2}$. Furthermore, the graph of a diagonal matrix is empty. 

In order to visualize graphs, we usually represent vertices with dots or circles in the plane, and if $v_i$ is adjacent to $v_j$, then we draw a line (or a curve) connecting $v_i$ to $v_j$. The {\em graph of} a real symmetric matrix $A \in \mathbb{R}^{n\times n}$ is a simple graph on $n$ vertices $1,2,\ldots,n$, and vertices $i$ and $j$ ($i\neq j$) are adjacent if and only if $a_{ij} \neq 0$. Note that the diagonal entries of $A$ have no role in this construction.

%%%%%%%%%%%%%%%%%%%%%%%%%%%%%%%%%%%%%%
\section{Examples and problem formulation}\label{Sec:ex}

We present two (connected) examples from mechanics. The first (Example \ref{ex:linearsystem}) is a fundamental case where masses, springs, and dampers are \textit{serially linked} together, and both ends are \textit{fixed}. The second one is a \textit{generally linked} system and is divided into two parts (Examples \ref{ex:firstgraph} and \ref{ex:systemofmassesandsprings}) and is from \cite{cdy07}.

\begin{example} \label{ex:linearsystem}
	Consider the \textit{serially linked} system of masses and springs sketched in Figure \ref{fig:pathsystemofmassesandsprings}. It is assumed that springs respond according to Hooke's law and that damping is negatively proportional to the velocity. All parameters $m,\,d,\,k$ are {\em positive}, and are associated with mass, damping, and stiffness, respectively.
	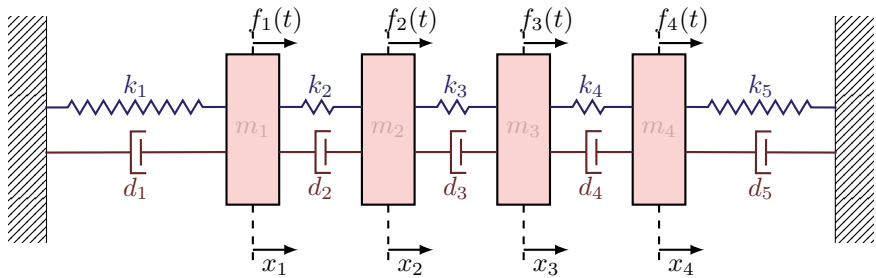
\begin{figure}[h]
	\begin{center}
	\begin{tikzpicture}[scale=.6]
 %styles 
 \tikzstyle{spring}=[color=dark-blue,thick,decorate,decoration={zigzag,pre length=0.3cm,post
 length=0.3cm,segment length=6}]
 \tikzstyle{damper}=[color=dark-red,thick,decoration={markings, mark connection node=dmp, mark=at position 0.5 with 
   {
     \node (dmp) [color=dark-red,thick,inner sep=0pt,transform shape,rotate=-90,minimum width=15pt,minimum height=3pt,draw=none] {};
     \draw [color=dark-red,thick] ($(dmp.north east)+(2pt,0)$) -- (dmp.south east) -- (dmp.south west) -- ($(dmp.north west)+(2pt,0)$);
     \draw [color=dark-red,thick] ($(dmp.north)+(0,-5pt)$) -- ($(dmp.north)+(0,5pt)$);
   }
 }, decorate]
 \tikzstyle{ground}=[fill,pattern=north east lines,draw=none,minimum width=6cm,minimum height=0.5cm]
  
 %walls
 \node (wall) [ground, rotate=-90, minimum width=3cm,yshift=-3cm] {};
 \draw (wall.north east) -- (wall.north west);
 
 \node (wall2) [ground, rotate=90, minimum width=3cm,yshift=-8cm] {};
 \draw (wall2.north east) -- (wall2.north west); 
 
 %masses
 \node[draw,outer sep=0pt,thick,fill = light-red,fill opacity=0.2, text opacity = 1] (M1) at (0,0)[minimum width=.4cm, minimum height=2cm] {$m_1$};
 \node[draw,outer sep=0pt,thick,fill = light-red,fill opacity=0.2, text opacity = 1] (M2) at (3,0) [minimum width=.4cm, minimum height=2cm] {$m_2$};
 \node[draw,outer sep=0pt,thick,fill = light-red,fill opacity=0.2, text opacity = 1] (M3) at (6,0) [minimum width=.4cm, minimum height=2cm] {$m_3$};
 \node[draw,outer sep=0pt,thick,fill = light-red,fill opacity=0.2, text opacity = 1] (M4) at (9,0) [minimum width=.4cm, minimum height=2cm] {$m_4$};
 
 %springs 
 \draw[spring] ($(wall.east) + (.4,3)$) -- ($(M1.west) + (0,.5)$) node [midway,above] {$k_1$};
 \draw[spring] ($(M1.east) + (0,.5)$) -- ($(M2.west) + (0,.5)$) node [midway,above] {$k_2$};
 \draw[spring] ($(M2.east) + (0,.5)$) -- ($(M3.west) + (0,.5)$) node [midway,above] {$k_3$};
 \draw[spring] ($(M3.east) + (0,.5)$) -- ($(M4.west) + (0,.5)$) node [midway,above] {$k_4$};
 \draw[spring] ($(M4.east) + (0,.5)$) -- ($(wall2.west) + (-.4,3)$) node [midway,above] {$k_5$};
 
 %dampings
 \draw[damper] ($(wall.east) + (.4,2)$) -- ($(M1.west) + (0,-.5)$) node [midway,below,yshift=-2mm] {$d_1$};
 \draw[damper] ($(M1.east) + (0,-.5)$) -- ($(M2.west) + (0,-.5)$) node [midway,below,yshift=-2mm] {$d_2$};
 \draw[damper] ($(M2.east) + (0,-.5)$) -- ($(M3.west) + (0,-.5)$) node [midway,below,yshift=-2mm] {$d_3$};
 \draw[damper] ($(M3.east) + (0,-.5)$) -- ($(M4.west) + (0,-.5)$) node [midway,below,yshift=-2mm] {$d_4$};
 \draw[damper] ($(M4.east) + (0,-.5)$) -- ($(wall2.west) + (-.4,2)$) node [midway,below,yshift=-2mm] {$d_5$};
 
 %forces
 \draw[thick, dashed] ($(M1.north)$) -- ($(M1.north) + (0,.5)$);
 \draw[thick, -latex] ($(M1.north) + (0,0.25)$) -- ($(M1.north) + (1,0.25)$) node [midway, above] {$f_1(t)$};
 \draw[thick, dashed] ($(M2.north)$) -- ($(M2.north) + (0,.5)$);
 \draw[thick, -latex] ($(M2.north) + (0,0.25)$) -- ($(M2.north) + (1,0.25)$) node [midway, above] {$f_2(t)$};
 \draw[thick, dashed] ($(M3.north)$) -- ($(M3.north) + (0,.5)$);
 \draw[thick, -latex] ($(M3.north) + (0,0.25)$) -- ($(M3.north) + (1,0.25)$) node [midway, above] {$f_3(t)$};
 \draw[thick, dashed] ($(M4.north)$) -- ($(M4.north) + (0,.5)$);
 \draw[thick, -latex] ($(M4.north) + (0,0.25)$) -- ($(M4.north) + (1,0.25)$) node [midway, above] {$f_4(t)$};
 
 %displacements
 \draw[thick, dashed] ($(M1.south)$) -- ($(M1.south) + (0,-1.25)$);
 \draw[thick, -latex] ($(M1.south) + (0,-1)$) -- ($(M1.south) + (1,-1)$) node [midway, below] {$x_1$};
 \draw[thick, dashed] ($(M2.south)$) -- ($(M2.south) + (0,-1.25)$);
 \draw[thick, -latex] ($(M2.south) + (0,-1)$) -- ($(M2.south) + (1,-1)$) node [midway, below] {$x_2$};
 \draw[thick, dashed] ($(M3.south)$) -- ($(M3.south) + (0,-1.25)$);
 \draw[thick, -latex] ($(M3.south) + (0,-1)$) -- ($(M3.south) + (1,-1)$) node [midway, below] {$x_3$};
 \draw[thick, dashed] ($(M4.south)$) -- ($(M4.south) + (0,-1.25)$);
 \draw[thick, -latex] ($(M4.south) + (0,-1)$) -- ($(M4.south) + (1,-1)$) node [midway, below] {$x_4$};
	\end{tikzpicture}
	\end{center}
	\caption{A four-degree-of-freedom serially linked mass-spring system.}
	\label{fig:pathsystemofmassesandsprings}
	\end{figure}
	
	There is a corresponding matrix polynomial
	\begin{equation} 
	A(z) = A_2 z^2 + A_1 z + A_0,\hspace*{.5cm} A_s \in \mathbb{R}^{4\times 4},\hspace*{.5cm} s = 0,1,2,
\end{equation}
	where \begin{small}
\begin{align}\label{tridiagonalmatrices}
	\begin{split}
	A_2 &= \left[ \begin{array}{cccc}
		m_1 & 0 & 0 & 0\\
		0 & m_2 & 0 & 0\\
		0 & 0 & m_3 & 0\\
		0 & 0 & 0 & m_4	
	\end{array} \right],\\	
	A_1 &= \left[ \begin{array}{cccc}
		d_1 + d_2 & {\color{dark-red}{{-d_2}}} & 0 & 0\\
		{\color{dark-red}{{-d_2}}} & d_2 + d_3 & {\color{dark-red}{{-d_3}}} & 0\\
		0 &{\color{dark-red}{{-d_3}}} & d_3 + d_4 & {\color{dark-red}{{-d_4}}}\\
		0 & 0 & {\color{dark-red}{{-d_4}}} & d_4 + d_5
	\end{array} \right],   \\
	A_0 &= \left[ \begin{array}{cccc}
		k_1 + k_2 & {\color{dark-blue}{{-k_2}}} & 0 & 0\\
		{\color{dark-blue}{{-k_2}}} & k_2 + k_3 & {\color{dark-blue}{{-k_3}}} & 0\\
		0 & {\color{dark-blue}{{-k_3}}} & k_3 + k_4 & {\color{dark-blue}{{-k_4}}}\\
		0 & 0 & {\color{dark-blue}{{-k_4}}} & k_4 + k_5	
	\end{array} \right].
	\end{split}
\end{align}
\end{small}

The graph of $A_2$ consists of four distinct vertices (it has no edges).
Because the $d$'s and $k$'s are all nonzero, the graphs of $A_0$ and $A_1$ coincide. For convenience, we name them $G$ and $H$ respectively (see Figure \ref{fig:graphsofmatricesdeeandkaytridiagonal}). 

\begin{figure}[h]
\begin{center}
\begin{tikzpicture}[scale=.6,node distance = 1.5cm]
	\begin{scope}[shift={(0,0)},scale=1]
	\node[] () at (-1.5,0) {$G:$};
	\node[draw,circle,fill = light-red,fill opacity=0.2, text opacity = 1] (1) at (0,0) {$1$};
	\node[draw,circle,fill = light-red,fill opacity=0.2, text opacity = 1] (2) [right of = 1] {$2$};
	\node[draw,circle,fill = light-red,fill opacity=0.2, text opacity = 1] (3) [right of = 2] {$3$};
	\node[draw,circle,fill = light-red,fill opacity=0.2, text opacity = 1] (4) [right of = 3] {$4$};
	\draw[very thick,color=dark-blue] (1) --node [midway, above] {$-k_2$} (2) --node [midway, above] {$-k_3$} (3) --node [midway, above] {$-k_4$} (4);
	\end{scope}
	
	\begin{scope}[shift={(0,-4)},scale=1]
	\node[] () at (-1.5,0) {$H:$};

	\node[draw,circle,fill = light-red,fill opacity=0.2, text opacity = 1] (1) at (0,0) {$1$};
	\node[draw,circle,fill = light-red,fill opacity=0.2, text opacity = 1] (2) [right of = 1] {$2$};
	\node[draw,circle,fill = light-red,fill opacity=0.2, text opacity = 1] (3) [right of = 2] {$3$};
	\node[draw,circle,fill = light-red,fill opacity=0.2, text opacity = 1] (4) [right of = 3] {$4$};
	\draw[very thick,color=dark-red] (1) --node [midway, above] {$-d_2$} (2) --node [midway, above] {$-d_3$} (3) --node [midway, above] {$-d_4$} (4);
	\end{scope}
\end{tikzpicture}
\end{center}
\caption{Graphs of $A_0$ and $A_1$ in Eq. \eqref{tridiagonalmatrices}.}
\label{fig:graphsofmatricesdeeandkaytridiagonal}
\qed
\end{figure}
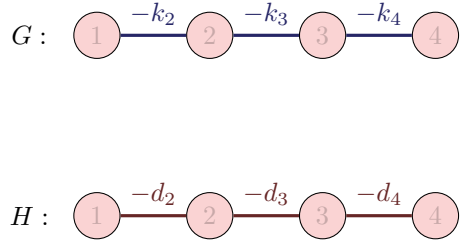 

In the later sections we will study how to perturb a diagonal matrix polynomial of degree two to
 achieve a new matrix polynomial, but the graphs of its coefficients are just those of this tridiagonal $A(z)$ (so that the physical structure of Figure \ref{fig:pathsystemofmassesandsprings} is maintained). In order to do this, we define matrices with variables on the diagonal entries and the nonzero entries of $A_0$ and $A_1$ in Eq.~\eqref{tridiagonalmatrices} as follows (where the diagonal entries of $A_s$ are $x_{sj}$'s and the off-diagonal entries are zero or $y_{sj}$'s). Thus, for $n=4$,

\begin{equation}
	A_0 = \left[ \begin{array}{cccc}
		x_{0,1} & y_{0,1} & 0 & 0\\
		y_{0,1} & x_{0,2} & y_{0,2} & 0\\
		0 & y_{0,2} & x_{0,3} & y_{0,3}\\
		0 & 0 & y_{0,3} & x_{0,4}
	\end{array} \right], \;\;\;
	A_1 = \left[ \begin{array}{cccc}
		x_{1,1} & y_{1,1} & 0 & 0\\
		y_{1,1} & x_{1,2} & y_{1,2} & 0\\
		0 & y_{1,2} & x_{1,3} & y_{1,3}\\
		0 & 0 & y_{1,3} & x_{1,4}
	\end{array} \right].
\end{equation}

\end{example}
 More generally, the procedure is given in Definition \ref{3.1}.\qed

In the next example we will, again, consider two graphs and their associated matrices and then, in Example \ref{ex:systemofmassesandsprings}, we see how they can be related to a physical network of masses and springs.

\begin{example} \label{ex:firstgraph}
	Define the (loopless) graph $G = (V_1,E_1)$ by $V_1 = \{ 1,2,3,4 \}$ with edges 
\begin{equation}
	 E_1 = \{ e_2 = \{1,2\},\; e_3 = \{2,3\},\; e_4 =  \{3,4\}, \; e_5 = \{1,3\} \},
\end{equation}
and the graph $H = (V_2,E_2)$ with $V_2 = \{ 1,2,3,4 \}$ and edges 
\begin{equation} 
	E_2 = \{ e_2 = \{1,3\},\; e_3 = \{3,4\} \}.
\end{equation}
Then we can visualize $G$ and $H$ as shown in Figure \ref{fig:graphsofmatricesdeeandkay}.

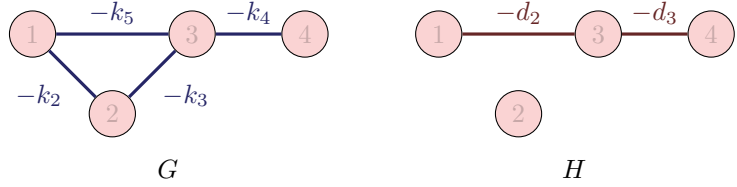
\begin{figure}[h]
\begin{center}
\begin{tikzpicture}[scale=.6,node distance = 1.5cm]
	\begin{scope}[shift={(0,0)},scale=1]
	\node[] () at (3,-3) {$G$};
	\node[draw,circle,fill = light-red,fill opacity=0.2, text opacity = 1] (1) at (0,0) {$1$};
	\node[draw,circle,fill = light-red,fill opacity=0.2, text opacity = 1] (2) [below right of = 1] {$2$};
	\node[draw,circle,fill = light-red,fill opacity=0.2, text opacity = 1] (3) [above right of = 2] {$3$};
	\node[draw,circle,fill = light-red,fill opacity=0.2, text opacity = 1] (4) [right of = 3] {$4$};
	\draw[very thick,color=dark-blue] (4) --node [midway, above] {$-k_4$} (3) --node [midway, below right] {$-k_3$} (2) --node [midway, below left] {$-k_2$} (1) --node [midway, above] {$-k_5$} (3);
	\end{scope}
	
	\begin{scope}[shift={(9,0)},scale=1]
	\node[] () at (3,-3) {$H$};
	\node[draw,circle,fill = light-red,fill opacity=0.2, text opacity = 1] (1) at (0,0) {$1$};
	\node[draw,circle,fill = light-red,fill opacity=0.2, text opacity = 1] (2) [below right of = 1] {$2$};
	\node[draw,circle,fill = light-red,fill opacity=0.2, text opacity = 1] (3) [above right of = 2] {$3$};
	\node[draw,circle,fill = light-red,fill opacity=0.2, text opacity = 1] (4) [right of = 3] {$4$};
	\draw[very thick,color=dark-red] (4) --node [midway, above] {$-d_3$} (3) --node [midway, above] {$-d_2$} (1);
	\end{scope}
\end{tikzpicture}
\end{center}
\caption{Graphs $G$ and $H$.}
\label{fig:graphsofmatricesdeeandkay}
\end{figure}   

Now define matrices $K$ and $D$ in Eq.~\eqref{eq1} as follows:
	\begin{align}\label{matricesdeekay}
	\begin{split}
	K &= \left[ \begin{array}{cccc}
		k_1 + k_2 + k_5 & {\color{dark-blue}{{-k_2}}} & {\color{dark-blue}{{-k_5}}} & 0\\
		{\color{dark-blue}{{-k_2}}} & k_2 + k_3 & {\color{dark-blue}{{-k_3}}} & 0\\
		{\color{dark-blue}{{-k_5}}} & {\color{dark-blue}{{-k_3}}} & k_3 + k_4 + k_5 & {\color{dark-blue}{{-k_4}}}\\
		0 & 0 & {\color{dark-blue}{{-k_4}}} & k_4	
	\end{array} \right],\\
	D &= \left[ \begin{array}{cccc}
		d_1 + d_2 & 0 & {\color{dark-red}{{-d_2}}} & 0\\
		0 & 0 & 0 & 0\\
		{\color{dark-red}{{-d_2}}} & 0 & d_2 + d_3 & {\color{dark-red}{{-d_3}}}\\
		0 & 0 & {\color{dark-red}{{-d_3}}} & d_3
	\end{array} \right]
	\end{split}
	\end{align}
	where all $d_i$ and $k_i$ are positive. It is easily seen that the graph of $K$ is $G$ of Figure \ref{fig:graphsofmatricesdeeandkay}, since $G$ is a graph on the $4$ vertices $1,2,3$, and $4$, and the $\{1,2\}$, $\{1,3\}$, $\{2,3\}$, and $\{3,4\}$ entries are all nonzero. Furthermore, $G$ has edges $\{1,2\}$, $\{1,3\}$, $\{2,3\}$, and $\{3,4\}$ corresponding to the nonzero entries of $K$. Similarly, one can check that the graph of $D$ is $H$.

Let $G$ and $H$ be the graphs shown in Figure \ref{fig:graphsofmatricesdeeandkay}, and let $D$ and $K$ be defined as in Eq.~\eqref{matricesdeekay}. Using Definition \ref{3.1}, we define matrices associated with the graphs:
\begin{equation}
	A_0 = \left[ \begin{array}{cccc}
		x_{0,1} & y_{0,1} & y_{0,2} & 0\\
		y_{0,1} & x_{0,2} & y_{0,3} & 0\\
		y_{0,2} & y_{0,3} & x_{0,3} & y_{0,4}\\
		0 & 0 & y_{0,4} & x_{0,4}
	\end{array} \right], \;\;\;
	A_1 = \left[ \begin{array}{cccc}
		x_{1,1} & 0 & y_{1,1} & 0\\
		0 & x_{1,2} & 0 & 0\\
		y_{1,1} & 0 & x_{1,3} & y_{1,2}\\
		0 & 0 & y_{1,2} & x_{1,4}
	\end{array} \right],
\end{equation}
so that
\begin{equation} 
	K = A_0(k_1+k_2+k_3,\; k_2+k_3,\; k_3+k_4+k_5,\; k_4,-k_2,-k_3,-k_4,-k_5), 
\end{equation}
\begin{equation}
	D= A_1(d_1+d_2,\; 0,\; d_2+d_3, \; d_3, \;-d_2,\; -d_3). \;\;\;\;  
	\qed
\end{equation} 
\end{example}

More generally, in this paper, structure is imposed on $L(z)$ in Eq.~\eqref{eq1} by requiring that $M$ is positive definite and diagonal, $D$ and $K$ are real and symmetric, and {\em nonzero entries in $D$ and $K$ are associated with the} connectivity {\em of nodes in a graph} - as illustrated above.

\begin{example}\label{ex:systemofmassesandsprings}
(See \cite{cdy07}.) A vibrating ``mass/spring'' system is sketched in Figure \ref{fig:systemofmassesandsprings}.
It is assumed that springs respond according to Hooke's law and that damping is negatively proportional to the velocity. 

The quadratic polynomial representing the dynamical equations of the system has the form Eq.~\eqref{eq1} with $n=4$. The coefficient matrices corresponding to this system are the diagonal matrix 
\begin{equation}
    M=\diag[m_1, m_2, m_3, m_4]  
\end{equation}
and matrices $D$ and $K$ in Eq.~\eqref{matricesdeekay}. 
It is important to note that (for physical reasons) the $m_i$, $d_i$, and $k_i$ parameters {\em are all positive}. 

\begin{figure}[h]
\begin{center}
 \begin{tikzpicture}[scale=.7]
styles 
\tikzstyle{spring}=[color=dark-blue,thick,decorate,decoration={zigzag,pre length=0.3cm,post
length=0.3cm,segment length=6}]

 \tikzstyle{damper}=[color=dark-red,thick,decoration={markings, mark connection node=dmp, mark=at position 0.5 with 
   {
     \node (dmp) [color=dark-red,thick,inner sep=0pt,transform shape,rotate=-90,minimum width=15pt,minimum height=3pt,draw=none] {};
     \draw [color=dark-red,thick] ($(dmp.north east)+(2pt,0)$) -- (dmp.south east) -- (dmp.south west) -- ($(dmp.north west)+(2pt,0)$);
     \draw [color=dark-red,thick] ($(dmp.north)+(0,-5pt)$) -- ($(dmp.north)+(0,5pt)$);
   }
 }, decorate]
 \tikzstyle{ground}=[fill,pattern=north east lines,draw=none,minimum width=6cm,minimum height=0.5cm]
  
 %wall
 \node (wall) [ground, rotate=-90, minimum width=3cm,yshift=-3cm] {};
 \draw (wall.north east) -- (wall.north west);
 %masses
 \node[draw,outer sep=0pt,thick,fill = light-red,fill opacity=0.2, text opacity = 1] (M1) at (0,0)[minimum width=.4cm, minimum height=5cm] {$m_1$};
 \node[draw,outer sep=0pt,thick,fill = light-red,fill opacity=0.2, text opacity = 1] (M2) at (3,-2.5) [minimum width=.4cm, minimum height=2cm] {$m_2$};
 \node[draw,outer sep=0pt,thick,fill = light-red,fill opacity=0.2, text opacity = 1] (M3) at (6,0) [minimum width=.4cm, minimum height=4cm] {$m_3$};
 \node[draw,outer sep=0pt,thick,fill = light-red,fill opacity=0.2, text opacity = 1] (M4) at (9,0) [minimum width=.4cm, minimum height=2cm] {$m_4$};
 %springs 
 \draw[spring] ($(wall.east) + (.35,2.7)$) -- ($(M1.west) + (0,.5)$) node [midway,above] {$k_1$};
 \draw[spring] ($(M1.east) + (0,-2)$) -- ($(M2.west) + (0,.5)$) node [midway,above] {$k_2$};
 \draw[spring] ($(M2.east) + (0,.5)$) -- ($(M3.west) + (0,-2)$) node [midway,above] {$k_3$};
 \draw[spring] ($(M3.east) + (0,.5)$) -- ($(M4.west) + (0,.5)$) node [midway,above] {$k_4$};
 \draw[spring] ($(M1.east) + (0,2.5)$) -- ($(M3.west) + (0,2.5)$) node [midway,above] {$k_5$};
 %dampings
 \draw[damper] ($(wall.east) + (.35,1.2)$) -- ($(M1.west) + (0,-1)$) node [midway,below,yshift=-2mm] {$d_1$};
 \draw[damper] ($(M1.east) + (0,1.7)$) -- ($(M3.west) + (0,1.7)$) node [midway,below,yshift=-2mm] {$d_2$};
 \draw[damper] ($(M3.east) + (0,-.5)$) -- ($(M4.west) + (0,-.5)$) node [midway,below,yshift=-2mm] {$d_3$};
 %forces
 \draw[thick, dashed] ($(M1.north)$) -- ($(M1.north) + (0,.5)$);
 \draw[thick, -latex] ($(M1.north) + (0,0.25)$) -- ($(M1.north) + (1,0.25)$) node [midway, above] {$f_1(t)$};
 \draw[thick, dashed] ($(M2.north)$) -- ($(M2.north) + (0,.5)$);
 \draw[thick, -latex] ($(M2.north) + (0,0.25)$) -- ($(M2.north) + (1,0.25)$) node [midway, above] {$f_2(t)$};
 \draw[thick, dashed] ($(M3.north)$) -- ($(M3.north) + (0,.5)$);
 \draw[thick, -latex] ($(M3.north) + (0,0.25)$) -- ($(M3.north) + (1,0.25)$) node [midway, above] {$f_3(t)$};
 \draw[thick, dashed] ($(M4.north)$) -- ($(M4.north) + (0,.5)$);
 \draw[thick, -latex] ($(M4.north) + (0,0.25)$) -- ($(M4.north) + (1,0.25)$) node [midway, above] {$f_4(t)$};
 %displacements
 \draw[thick, dashed] ($(M1.south)$) -- ($(M1.south) + (0,-2)$);
 \draw[thick, -latex] ($(M1.south) + (0,-1.75)$) -- ($(M1.south) + (1,-1.75)$) node [midway, below] {$x_1$};
 \draw[thick, dashed] ($(M2.south)$) -- ($(M2.south) + (0,-1.5)$);
 \draw[thick, -latex] ($(M2.south) + (0,-1.25)$) -- ($(M2.south) + (1,-1.25)$) node [midway, below] {$x_2$};
 \draw[thick, dashed] ($(M3.south)$) -- ($(M3.south) + (0,-2.5)$);
 \draw[thick, -latex] ($(M3.south) + (0,-2.25)$) -- ($(M3.south) + (1,-2.25)$) node [midway, below] {$x_3$};
 \draw[thick, dashed] ($(M4.south)$) -- ($(M4.south) + (0,-3.5)$);
 \draw[thick, -latex] ($(M4.south) + (0,-3.25)$) -- ($(M4.south) + (1,-3.25)$) node [midway, below] {$x_4$};
\end{tikzpicture}
\end{center}
\caption{A four-degree-of-freedom mass-spring system.}
\label{fig:systemofmassesandsprings}
\end{figure}
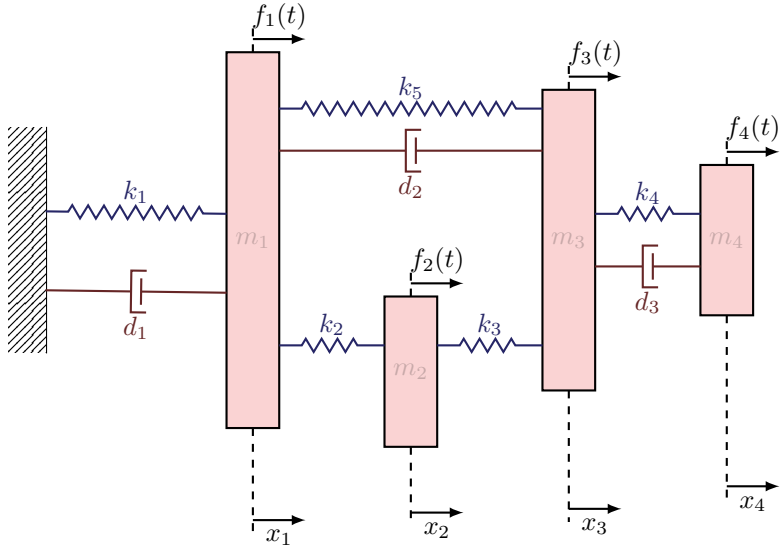

Consider the corresponding system in Eq.~\eqref{eq1} together with matrices in Eq.~\eqref{matricesdeekay}. The graphs of $K$ and $D$ are, respectively, $G$ and $H$ in Figure \ref{fig:graphsofmatricesdeeandkay}. Note that the two edges of graph $H$ correspond to the two dampers between the masses (that is, dampers $d_2$ and $d_3$), and the four edges of $G$ correspond to the springs between the masses (with constants $k_2$, \ldots, $k_5$) in Figure \ref{fig:systemofmassesandsprings}. In contrast, $d_1$ and $k_1$ contribute to just one  diagonal entry of $L(z)$.       \qed
\end{example}   

Using the ideas developed above we study the following more general problem:

\noindent{\bf A Structured Inverse Quadratic Problem:}\\ 
For a given set of $2n$ real numbers, $\Lambda$, and given graphs $G$ and $H$ on $n$ vertices, do there exist {\em real symmetric} matrices $M, D, K \in \mathbb{R}^{n \times n}$ such that the set of proper values of $L(z) = M z^2 + D z + K$ is $\Lambda$, $M$ is diagonal and positive definite, the graph of $D$ is $H$, and the graph of $K$ is $G$? (Note, in particular, that the constructed systems are to have {\em entirely real spectrum}.)
\medskip

More generally, we study problems of this kind of higher degree - culminating in Theorem \ref{thm:main}. A partial answer to the ``quadratic'' problem is provided in Corollary \ref{lcor}. In particular, it will be shown that a solution exists when the given proper values are all distinct. The strategy is to start with a diagonal matrix polynomial with the given proper values, and then perturb the off diagonal entries of the coefficient matrices so that they realize the given graph structure. In doing so the proper values change. Then we argue that there is an adjustment of the diagonal entries so that the resulting matrix polynomial has the given proper values. The last step involves using the implicit function theorem. Consequently, all the perturbations are small and the resulting matrix is \textit{close} to a diagonal matrix. We solve the problem for matrix polynomials of general degree, $k$, and the quadratic problem is the special case $k=2$.

The authors of \cite{cdy07} deal with an inverse problem in which the graphs $G$ and $H$ are \emph{paths}. That is, the corresponding matrices to be reconstructed are \emph{tridiagonal} matrices where the superdiagonal and subdiagonal entries are nonzero as in Example \ref{ex:linearsystem} (but not Example \ref{ex:firstgraph}). In this particular problem only a few proper values and their corresponding proper vectors are given. For more general graphs, it is argued that ``the issue of solvability is problem dependent and has to be addressed structure by structure.'' This case, in which the graphs of the matrices are arbitrary and only a few proper values and their corresponding proper vectors are given, is considered in \cite{dlc09,ldc10-1,ldc10-2}.

%%%%%%%%%%%%%%%%%%%%%%%%%%%%%%%%%%%%%%%%%%%
\section{The higher degree problem} \label{sec:hd}

The machinery required for the solution of our inverse quadratic problems is readily extended for use in the context of problems of higher degree. So we now focus on polynomials $A(z)$ of general degree $k \geq 1$ with
$A_0, A_1, \ldots, A_k \in \mathbb{R}^{n \times n}$ and symmetric. With $z \in \mathbb{C}$, the polynomials have the form
	\begin{equation} \label{eqa} 
		A(z) := A_k z^k + \cdots + A_1 z + A_0,\;\;\;\;A_k \ne 0,  
	\end{equation}
and we write
	\begin{equation} \label{eqaderivative} 
		A^{(1)}(z) = k A_k z^{k-1} + \cdots + 2 A_2 z + A_1.
	\end{equation}
Since $A_k \neq 0$, the matrix polynomial $A(z)$ is said to have \emph{degree} $k$. If $\det A(z)$ has an isolated zero at $z_0$ of multiplicity $m$, then $z_0$ is a proper value of $A(z)$ of \emph{algebraic multiplicity} $m$. A proper value with $m = 1$ is said to be \emph{simple}. 
	
	If $z_0$ is a proper value of $A(z)$ and the null space of $A(z_0)$ has dimension $r$, then $z_0$ is a proper value of $A(z)$ of \emph{geometric multiplicity} $r$. If $z_0$ is a proper value of $A(z)$ and its algebraic and geometric multiplicities agree, then the proper value $z_0$ is said to be \emph{semisimple}.

We assume that {\em all} the proper values and graph structures associated with $A_0,\ldots,A_k$ are given (as in Eq.~\eqref{tridiagonalmatrices}, where $k=2$). We are concerned only with the solvability of the problem. In particular, we show that {\em when all the proper values are real and simple, the structured inverse quadratic problem is solvable for any given graph-structure}. The constructed matrices, $A_0,A_1,\ldots,A_k$, will then be real and symmetric. More generally, our approach shows the existence of an \emph{open} set of solutions for polynomials of {\em any degree} and the important quadratic problem (illustrated above) is a special case. Consequently, this shows that the solution is not unique.

The techniques used here are generalizations of those appearing in \cite{hs13}, where the authors show the existence of a solution for the {\em linear} structured inverse eigenvalue problem. A different generalization of these techniques is used in \cite{h17} to solve the {\em linear} problem when the solution matrix is not necessarily symmetric, and this admits complex conjugate pairs of eigenvalues. 

First consider a {\em diagonal} matrix polynomial with some given proper values. The graph of each (diagonal) coefficient of the matrix polynomial is, of course, a graph with vertices but no edges (an empty graph). We suppose that such a graph is assigned for each coefficient. We perturb the off-diagonal entries (corresponding to the edges of the graphs) to nonzero numbers in such a way that the new matrix polynomial has given graphs (as with $G$ and $H$ in Examples \ref{ex:linearsystem} and \ref{ex:firstgraph}). Of course, this will change the proper values of the matrix polynomial. Then we use the implicit function theorem to show that if the perturbations of the diagonal system are small, the diagonal entries can be adjusted so that the resulting matrix polynomial has the same proper values as the unperturbed diagonal system. 

In order to use the implicit function theorem, we need to compute the derivatives of a proper value of a matrix polynomial with respect to perturbations of one entry of one of the coefficient matrices. That will be done in this section. Then, in Section \ref{sec:diagonalpert}, we construct a diagonal matrix polynomial with given proper values and show that a function that maps matrix polynomials to their proper values has a nonsingular Jacobian at this diagonal matrix. In Section \ref{sec:existencetheorem}, the implicit function theorem is used to establish the existence of a solution for the structured inverse problem.

%%%%%%%%%%%%%%%%%%%%%%%%%%%%%%%%%%
\subsection{Symmetric perturbations of diagonal systems}
	
Now let us focus on matrix polynomials $A(z)$  of degree $k$ with {\em real} and {\em diagonal} coefficients. The next lemma provides the derivative of a simple proper value of $A(z)$ when the diagonal $A(z)$ is subjected to a {\em real symmetric} perturbation. Thus, we consider
\begin{equation}  \label{eq.cab} C(z,t):=A(z) + tB(z)   \end{equation}
where $t\in \mathbb{R}$, $|t| < \varepsilon$ for some $\varepsilon > 0$, and 
\begin{equation} B(z) = B_k z^k +  B_{k-1}z^{k-1} + \cdots + B_1z + B_0 \end{equation} 
with $B_s^T=B_s \in \mathbb{R}^{n\times n}$ for $s=0,1,2,\ldots,k$. 

Let us denote the derivative of a variable $c$ with respect to the perturbation parameter $t$ by $\dot{c}$. Also, let $\bm e_r \in \mathbb{R}^n$ be the $r$th column of the identity matrix (i.e. it has a $1$ in the $r$th position and zeros elsewhere). The following lemma is well-known. A proof is provided for expository purposes.

	\begin{lemma}[See Lemma 1 of \cite{lmz03}] \label{lem}
		Let $k$ and $n$ be fixed positive integers and let $A(z)$ in Eq.~\eqref{eqa} have real, diagonal, coefficients and a simple proper value $z_0$. Let $z(t)$ be the unique (necessarily simple) proper value of $C(z,t)$ in Eq.~\eqref{eq.cab} for which $z(t) \to z_0$ as $t \to 0$. Then there is an $r \in \{1,2,\ldots, n\}$ for which 
		\begin{equation} \label{eq:lba}
	\dot{z}(0) = -\frac{ \left( B(z_0) \right)_{rr}}{\left( A^{(1)}(z_0) \right)_{rr}}.
		\end{equation}
	\end{lemma}
	\begin{proof}
First observe that, because $z_0$ is a {\em simple} proper value of $A(z)$, there exists an analytic function of proper values $z(t)$ for $C(z,t)$ defined on a neighbourhood of $t=0$ for which $z(t)\rightarrow z_0$ as $t\rightarrow 0$. Furthermore, there is a corresponding differentiable proper vector $\bm v(t)$ of $C(z,t)$ for which $\bm v(t) \to \bm e_r$ for some $r=1,2,\ldots,n$, as $t \to 0$ (See Lemma 1 of \cite{lmz03}, for example). Thus, in a neighbourhood of $t=0$ we have
		\begin{equation}\label{eq:evalofc}
			C(z(t),t) \bm v(t) = \big( A(z)+tB(z)\big) \bm v(t) = \bm 0.
		\end{equation}
Then observe that
		\begin{align*}
			\frac{\rm d}{{\rm d}t} \left( z^j(t) (A_j + t B_j) \right) \at_{t = 0} 
			& = j z^{j-1}(t) \dot{z}(t) (A_j + t B_j) + z^j(t) B_j \at_{t = 0} \\
			& = j z_0^{j-1} \dot{z}(0) A_j + z_0^j B_j.
		\end{align*}
Thus, taking the first derivative of Eq.~\eqref{eq:evalofc} with respect to $t$ and then setting $t = 0$ we have $\bm v(0)=e_r$ and
		\begin{equation} \label{eq:derivativeofc}
			\left( (A^{(1)}(z_0) \dot{z}(0) + B(z_0) \right) \bm e_r 
			+
			A(z_0) \dot{\bm v}(0) = \bm 0.
		\end{equation}
Multiply by $\bm e_r^\top$ from the left to get
		\begin{equation}\label{eq:ff}
			\bm e_r^\top A^{(1)}(z_0) \dot{z}(0) \bm e_r 
			+
			\bm e_r^\top B(z_0) \bm e_r
			+
			\bm e_r^\top A(z_0) \dot{\bm v}(0) = 0.
		\end{equation}
But $\bm e_r^\top$ is a left proper vector of $A(z_0)$ corresponding to the proper value $z_0$. Thus, $\bm e_r^\top A(z_0) = \bm 0^\top$, and (\ref{eq:lba}) follows from (\ref{eq:ff}).
\end{proof}

Now we can calculate the changes in a simple proper value of $A(z)$ when an entry of just one of the coefficients, $A_s,$ is perturbed -- while maintaining symmetry. 
\begin{definition} \label{def:eeeyejay}
	For $1 \le i,j \le n$, define the symmetric $n\times n$ matrices $E_{ij}$ with:\\
	(a) exactly one nonzero entry, $e_{ii}=1$, when $j=i$, and \\
	(b) exactly two nonzero entries, $e_{ij}=e_{ji}=1$, when $j \ne i$. 
\end{definition} 

We perturb certain entries of $A(z)$ in Eq.~\eqref{eqa} (maintaining symmetry) by applying Lemma \ref{lem} with $B(z)=z^m E_{ij}$ to obtain:

	\begin{corollary}\label{cor:standardunitevec}
		Let $A(z)$ in Eq.~$\eqref{eqa}$ be diagonal with a simple proper value $z_0$ and corresponding unit proper vector $\bm e_r$. Let $z_m(t)$ be the proper value of the perturbed system $A(z)+t(z^mE_{ij})$, for some $i, j \in \{1,2,\ldots,n\}$, that approaches $z_0$ as $t \to 0$.
		Then
	\begin{equation} 
		\dot{z}_m(0) = 
			\begin{cases}
		\dfrac{-z_0^m}{\left( A^{(1)}(z_0) \right)_{rr}} & \text{when}\; \; r = i = j,\\
				0 &  \text{when}\; i \ne j.
			\end{cases}
	\end{equation}
	\end{corollary}
	Note also that, when we perturb {\em off-diagonal} entries of the diagonal matrix function 
$A(z)$ in Eq.~$\eqref{eqa}$, we obtain $\dot{z}_m(0) = 0$.

%%%%%%%%%%%%%%%%%%%%%%%%%%%%%%%%%%%%%%%%%%
\section{A special diagonal matrix polynomial} \label{sec:diagonalpert}

\subsection{Construction}
We construct an $n\times n$ real diagonal matrix polynomial $A(z)$ of degree $k$, with given real proper values $\lambda_1,\lambda_2,\ldots,\lambda_{nk}$. Then (see Eq.~\eqref{functioneff}) we define a function $f$ that maps the entries of $A(z)$ to its proper values and show that the Jacobian of $f$ when evaluated at the constructed $A(z)$ is nonsingular. This construction prepares us for use of the implicit function theorem in the proof of the main result in the next section.

\textbf{Step 1:} Let $[k]_r$ denote the sequence of $k$ integers $\{ (r-1)k+1, (r-1)k+2, \ldots, rk \}$, for $r=1,2,\ldots,n$. Thus, $[k]_1 = \{ 1,2,\ldots, k \}$, $[k]_2 = \{ k+1, k+2, \ldots, 2k \}$, and $[k]_n = \{ (n-1)k+1,(n-1)k+2, \ldots, nk \}$. We are to define an $n\times n$ diagonal matrix polynomial $A(z)$ where, for $i=1,2,\ldots,n,$ the zeros of the $i$-th diagonal entry are exactly those proper values $\lambda_q$ of $A(z)$ with $q \in [k]_i$. 

\textbf{Step 2:} Let $\alpha_{k,1},\ldots,\alpha_{k,n}$ be assigned positive numbers. 
We use these numbers to define the $n$ diagonal entries for each of $k$ diagonal matrix polynomials (of size $n\times n$). Then, for $s = 0,1,\ldots,k-1$, and $t=1,2,\ldots, n$ we define 
\begin{equation} \label{eq:wow}
	\alpha_{s,t} = (-1)^{k-s} \alpha_{k,t} \sum_{\substack{Q \subseteq [k]_t  \\ |Q| = k-s }} \prod_{q \in Q} \lambda_q. \;\;\;\; 
\end{equation}
Thus, the summation is over all subsets of size $k-s$ of the set of integers $[k]_t$. 

Now define
\begin{equation} \label{eq:vic}
	A_s: = \left[ \begin{array}{cccc}
	\alpha_{s,1} & 0 & \cdots & 0\\
	0 & \alpha_{s,2} & \cdots & 0\\
	\vdots & \vdots & \ddots & \vdots \\
	0 & \cdots & 0 & \alpha_{s,n}
	\end{array}\right] \text{ for } s = 0,1,\ldots,k, 
\end{equation}
and the diagonal matrix polynomial
\begin{equation}
	A(z): = \sum_{s=0}^{k} A_s z^s.
\end{equation}
Using (\ref{eq:wow}) and the fact that $\alpha_{k,j} \neq 0$ for each $j$, we see that
\begin{equation} \label{diagonalexample}
	A(z) = \left[ \begin{array}{cccc}
	\alpha_{k,1} \displaystyle\prod_{q\in [k]_1} (z-\lambda_q)  & 0 & \cdots & 0 \\
	0 & \alpha_{k,2} \displaystyle\prod_{q\in [k]_2} (z-\lambda_q) & \cdots & 0 \\
	\vdots & \vdots & \ddots & \vdots \\
	0 & 0 & \cdots & \alpha_{k,n} \displaystyle\prod_{q\in [k]_n} (z-\lambda_q) 
	\end{array}\right]
\end{equation}
has degree $k$, and the assigned proper values are $\lambda_1,\lambda_2,\ldots,\lambda_{nk}$.  Note that the proper vector corresponding to $\lambda_q$ is $e_r$ for $q \in [k]_r$. This completes our construction.

In the following theorem we use Corollary \ref{cor:standardunitevec} to examine perturbations of \emph{either} a diagonal entry $(i,i)$ of $A(z)$ in Eq.~\eqref{diagonalexample}, \emph{or} two of the (zero) off-diagonal entries, $(i,j)$ and $(j,i)$, of $A(z)$. 

\begin{theorem} \label{lem:derivativewrtoffdiagonals}
	Let $\lambda_1,\lambda_2,\ldots,\lambda_{nk}$ be $nk$ distinct real numbers, and let $A(z)$ be defined as in Eq.~\eqref{diagonalexample}. For a fixed $m \in \{ 0,1, \ldots, k-1\}$ and with $E_{ij}$ as in Definition $\ref{def:eeeyejay}$, define 
\begin{align*}
	P_m^{i,j}(z,t) &= A(z) +  z^m t E_{ij}.
\end{align*}

If $1 \le q \le nk$, and $\lambda_{q,m}^{i,j}(t)$ is the proper value of $P_m^{i,j}(z,t)$ that tends to $\lambda_q$ as $t \rightarrow 0$, then 
\begin{equation}\label{derivativeatij}
	\left( \dfrac{\partial \lambda_{q,m}^{i,j}(t)}{\partial t} \right)_{t=0} = \begin{cases}
		\dfrac{- \lambda_q^m}{A^{(1)} (\lambda_q)_{rr}}, &\text{ if } i = j = r \text{ and } q \in [k]_r,\\
		0, &\text{otherwise}.
	\end{cases}
\end{equation} 
\end{theorem}
\begin{proof}
It follows from the definition in Eq.~\eqref{diagonalexample} that $\det A^{(1)}(\lambda_q) \neq 0$ for all $q = 1,2,\ldots,nk$. That is, $ A^{(1)}(\lambda_q)_{rr} \neq 0$, for $r=1,2\ldots,n$. Then Eq.~\eqref{derivativeatij} follows from Corollary \ref{cor:standardunitevec}.
\end{proof}

%%%%%%%%%%%%%%%%%%%%%%%%%%%%%%%%%%%%%%%%
\subsection{The role of graphs}
We are going to construct matrices with variable entries, in order to adapt Corollary \ref{cor:standardunitevec} to the case when the entries of the $n\times n$ diagonal matrix $A$ in Eq. \eqref{diagonalexample} are independent variables. A small example of such a matrix appears in Example \ref{ex:firstgraph}. 

Let $G_0,G_1,\cdots,G_{k-1}$ be $k$ graphs on $n$ vertices and, for $0 \le s \le {k-1}$, let $G_s$ have $m_s$ edges $\{ i_\ell, j_\ell \}_{\ell=1}^{m_s}$ ($k=2$ and $n=4$ in Example 2.2). Define $2k$ vectors (2 per graph):
	\begin{equation} 
		\bm x_s = (x_{s,1}, \ldots, x_{s,n}) \in \mathbb{R}^n,
		\;\;\;\;
		\bm y_s = (y_{s,1},\ldots,y_{s,m_s}) \in \mathbb{R}^{m_s},
                \;\;\;\;s=0,1, \ldots , k-1,
    \end{equation}
and let $m = m_0 + m_1 + \cdots + m_{k-1}$ be the total number of the edges of all $G_s$.
(See Figure \ref{fig:graphsofmatricesdeeandkay}, where $k=2$ and $n=4$.)

\begin{definition}\label{3.1} 
(The matrix of a graph - see Example \ref{ex:firstgraph}) For $s=0,1,\cdots,k-1$, let $M_s = M_s(\bm x_s, \bm y_s)$ be an $n\times n$ symmetric matrix whose diagonal $(i,i)$ entry is $x_{s,i}$, the off-diagonal $(i_\ell, j_\ell)$ and $(j_\ell, i_\ell)$ entries are $y_{s,\ell}$ where $\{x_{i_{\ell}},x_{j_{\ell}}\}$ are edges of the graph $G_s$, and all other entries are zeros. We say that $M_s$ is the matrix of the graph $G_s$.
\end{definition}

Now let $A_k$ be the $n\times n$ diagonal matrix in Eq.~\eqref{eq:vic} (the leading coefficient of $A(z)$) and, using Definition \ref{3.1}, define the $n\times n$ matrix polynomial   
\begin{equation}\label{eq.mm}
	M = M(z, \bm x, \bm y) := z^k A_k + \sum_{s=0}^{k-1} z^s M_s(\bm x_s, \bm y_s), 
\end{equation}
where $\bm x = (\bm x_0, \ldots, \bm x_{k-1}) \in \mathbb{R}^{kn}$ and $\bm y = (\bm y_0, \ldots, \bm y_{k-1}) \in \mathbb{R}^{km_s}$. Thus, the coefficients of the matrix polynomial $M(z,\bm x, \bm y)$ are defined in terms of $k$ graphs, $G_s$, each having $n$ vertices and $m_s$ edges, for $s = 0,1,\ldots,k-1$. Note that, with the definition of the diagonal matrix polynomial $A(z)$ in (\ref{diagonalexample}), we have
\begin{equation}\label{matrixemmata}
	A(z) = M(z, \bm \alpha_0, \bm \alpha_1, \ldots, \bm \alpha_{k-1}, \bm 0, \bm 0, \ldots, \bm 0),
\end{equation}
where $\bm \alpha_s = (\alpha_{s,1}, \alpha_{s,2}, \ldots, \alpha_{s,n})$, for each $s = 0,1,\ldots,k-1$.

Recall that the strategy is to 
\begin{enumerate}
	\item[a)] perturb those {\em off-diagonal} (zero) entries of the diagonal matrix $A(z)$ in Eq.~\eqref{diagonalexample} that correspond to edges in the given graphs $G_s$ to small nonzero numbers, and then \item[b)] adjust the {\em diagonal} entries of the new matrix so that the proper values of the final matrix coincide with those of $A(z)$. 
\end{enumerate}
In order to do so, we keep track of the proper values of the matrix polynomial $M$ in Eq.~\eqref{eq.mm} by defining the following function:
\begin{align} \label{functioneff}
	f \colon \mathbb{R}^{kn+m} & \to \mathbb{R}^{kn} \nonumber\\
	(\bm x, \bm y) & \mapsto \left( \lambda_1(M), \lambda_2(M), \ldots, \lambda_{kn}(M) \right),
\end{align}
where $\lambda_q(M)$ is the $q$-th smallest proper value of $M(z, \bm x, \bm y)$. 

In order to show that, after small perturbations of the off-diagonal entries of $A(z)$, its proper values can be recovered by adjusting the diagonal entries, we will make use of a version of the implicit function theorem (stated below as Theorem \ref{IFT}). But in order to use the implicit function theorem, we will need to show that the Jacobian of the function $f$ in (\ref{functioneff}) is nonsingular at $A(z)$. 

Let $\jac_x(f)$ denote the submatrix of the Jacobian matrix of $f$ containing only the columns corresponding to the derivatives with respect to $x$ variables. Then $\jac_x(f)$ is 
\begin{equation}
\left[ \begin{array}{c|c|c}
\begin{array}{ccc}
\frac{\partial \lambda_1}{\partial x_{0,1}} & \cdots & \frac{\partial \lambda_1}{\partial x_{0,n}}\\
\vdots & \ddots & \vdots\\
\frac{\partial \lambda_k}{\partial x_{0,1}} & \cdots & \frac{\partial \lambda_k}{\partial x_{0,n}}
\end{array} & 
\begin{array}{ccc}
&&\\
&\cdots&\\
&&
\end{array}
& \begin{array}{ccc}
\frac{\partial \lambda_1}{\partial x_{k-1,1}} & \cdots & \frac{\partial \lambda_1}{\partial x_{k-1,n}}\\
\vdots & \ddots & \vdots\\
\frac{\partial \lambda_k}{\partial x_{k-1,1}} & \cdots & \frac{\partial \lambda_k}{\partial x_{k-1,n}}
\end{array}\\ \hline
&&\\
\vdots & \ddots & \vdots \\ 
&&\\\hline
\begin{array}{ccc}
\frac{\partial \lambda_{(n-1)k+1}}{\partial x_{0,1}} & \cdots & \frac{\partial \lambda_{(n-1)k+1}}{\partial x_{0,n}}\\
\vdots & \ddots & \vdots\\
\frac{\partial \lambda_{nk}}{\partial x_{0,1}} & \cdots & \frac{\partial \lambda_{nk}}{\partial x_{0,n}}
\end{array} & \cdots & \begin{array}{ccc}
\frac{\partial \lambda_{(n-1)k+1}}{\partial x_{k-1,1}} & \cdots & \frac{\partial \lambda_{(n-1)k+1}}{\partial x_{k-1,n}}\\
\vdots & \ddots & \vdots\\
\frac{\partial \lambda_{nk}}{\partial x_{k-1,1}} & \cdots & \frac{\partial \lambda_{nk}}{\partial x_{k-1,n}}
\end{array}
\end{array} \right],
\end{equation}
where each block is $k \times n$, and there are $n$ block rows and $k$ block columns. Note that, for example, the $(1,1)$ entry of $\jac_x(f)$ is the derivative of the smallest proper value of $M$ with respect to the variable in the $(1,1)$ position of $M_0$, and similarly the $(nk,nk)$ entry of $\jac_x(f)$ is the derivative of the largest proper value of $M$ with respect to the variable in the $(n,n)$ position of $M_{k-1}$. 

Then, using Theorem \ref{lem:derivativewrtoffdiagonals} we obtain:

\begin{corollary} \label{cor:paritalderivatives}
	Let $A(z)$ be defined as in Eq.~\eqref{diagonalexample}. Then
\begin{equation}
	\dfrac{\partial \lambda_{q}}{\partial x_{s,r}} \at_{A(z)} = \begin{cases}
		\dfrac{- \lambda_q^s}{\left(A^{(1)} (\lambda_q)\right)_{rr}}, &\text{ if } q \in [k]_r,\\
		0, &\text{otherwise}.
	\end{cases}
\end{equation}
\end{corollary}
\begin{proof}
	Note that the derivative is taken with respect to $x_{s,r}$. That is, with respect to the $(r,r)$ entry of the coefficient of $z^s$. Thus, using the terminology of Theorem \ref{lem:derivativewrtoffdiagonals}, the perturbation to consider is $P_s^{rr}(z,t)$. Then 
	\begin{equation}
		\left( \dfrac{\partial \lambda_{q,s}^{r,r}(t)}{\partial t} \right)_{t=0} =  \begin{cases}
		\dfrac{- \lambda_q^s}{\left(A^{(1)} (\lambda_q)\right)_{rr}}, &\text{ if } q \in [k]_r,\\
		0, &\text{otherwise}.
	\end{cases}
	\end{equation}
\end{proof}

The main result of this section is as follows:
\begin{theorem}\label{thm:nonsigularjacobian}
	Let $A(z)$ be defined as in Eq.~\eqref{diagonalexample}, and $f$ be defined by Eq.~\eqref{functioneff}. Then
	$\jac_x(f) \at_{A(z)}$ is nonsingular.
\end{theorem}

\begin{proof}
Corollary \ref{cor:paritalderivatives} implies that $\jac_x(f) \at_{A(z)}$ is 
\begin{equation}
J = - \left[ \begin{array}{c|c|c}
\begin{array}{cccc}
\frac{1}{\left(A^{(1)}(\lambda_1)\right)_{11}} & 0 & \cdots & 0\\
\vdots & \vdots & \ddots & \vdots\\
\frac{1}{\left(A^{(1)}(\lambda_k)\right)_{11}} & 0 & \cdots & 0
\end{array} & 
\begin{array}{ccc}
&&\\
&\cdots&\\
&&
\end{array}
& \begin{array}{cccc}
\frac{\lambda_1^{k-1}}{\left(A^{(1)}(\lambda_1)\right)_{11}} & 0 & \cdots & 0\\
\vdots & \vdots & \ddots & \vdots\\
\frac{\lambda_k^{k-1}}{\left(A^{(1)}(\lambda_k)\right)_{11}} & 0 & \cdots & 0
\end{array}\\ \hline
&&\\
\vdots & \ddots & \vdots \\ 
&&\\\hline
\begin{array}{cccc}
0 & \cdots & 0 & \frac{1}{\left(A^{(1)}(\lambda_{(n-1)k+1})\right)_{nn}}\\
\vdots & \ddots & \vdots & \vdots\\
0 & \cdots & 0 & \frac{1}{\left(A^{(1)}(\lambda_{nk})\right)_{nn}}
\end{array} & \cdots & \begin{array}{cccc}
0 & \cdots & 0 & \frac{\lambda_{(n-1)k+1}^{k-1}}{\left(A^{(1)}(\lambda_{(n-1)k+1})\right)_{nn}}\\
\vdots & \ddots & \vdots & \vdots\\
0 & \cdots & 0 & \frac{\lambda_{nk}^{k-1}}{\left(A^{(1)}(\lambda_{nk})\right)_{nn}}
\end{array}
\end{array} \right].
\end{equation}
Multiply $J$ by $-1$, and multiply row $q$ of $J$ by $\left(A^{(1)}(\lambda_q)\right)_{rr}$, for $q = 1,2,\ldots,kn$, and for the corresponding $r$, then reorder the columns to get
\begin{equation}
\left[ \begin{array}{c|c|c}
\begin{array}{cccc}
1 & \lambda_1 & \cdots & \lambda_1^{k-1}\\
1 & \lambda_2 & \cdots & \lambda_2^{k-1}\\
\vdots & \vdots & \ddots & \vdots\\
1 & \lambda_k & \cdots & \lambda_k^{k-1}\\
\end{array} & \begin{array}{ccc} 
& \cdots & \end{array} & O  \\ \hline
&&\\
\vdots & \ddots & \vdots\\
&&\\ \hline
O & \cdots  & \begin{array}{cccc}
1 & \lambda_{(n-1)k+1} & \cdots & \lambda_{(n-1)k+1}^{k-1}\\
1 & \lambda_{(n-1)k+2} & \cdots & \lambda_{(n-1)k+2}^{k-1}\\
\vdots & \vdots & \ddots & \vdots\\
1 & \lambda_{nk} & \cdots & \lambda_{nk}^{k-1}\\
\end{array}
\end{array} \right]
,
\end{equation}
which is a block diagonal matrix where each diagonal block is an invertible Vandermonde matrix since the $\lambda$'s are all distinct. Hence $J$ is nonsingular.
\end{proof}

%%%%%%%%%%%%%%%%%%%%%%%%%%%%%%%%%

\section{Existence Theorem} \label{sec:existencetheorem}
Now we use a version of the implicit function theorem to establish the existence of a solution for the structured inverse proper value problem (see \cite{dr14,kran02}). 

\begin{theorem}\label{IFT}
Let $F: \mathbb{R}^{s+r} \rightarrow \mathbb{R}^s$ be a continuously differentiable function on an open subset $U$ of $\mathbb{R}^{s+r}$ defined by 
	\begin{equation}
		F(\bm x, \bm y)=(F_1(\bm x, \bm y), F_2(\bm x, \bm y), \ldots, F_s(\bm x, \bm y)),
	\end{equation}
	where $\bm x = (x_1, \ldots, x_s) \in \mathbb{R}^s$ and $\bm y \in \mathbb{R}^r$. Let $(\bm a, \bm b)$ be an element of $U$ with $\bm a\in \mathbb{R}^s$ and $\bm b \in \mathbb{R}^r$,  and $\bm c$ be an element of $\mathbb{R}^s$ such that $F(\bm a, \bm b) = \bm c$. If 
	\begin{equation}
		\left[\frac{\partial F_i}{\partial x_{j}}\;{\rule[-3.6mm]{.1mm} {8mm}}_{(\bm a, \bm b)} \right]
	\end{equation}
is nonsingular, then there exist an open neighbourhood $V$ of $\bm a$ and an open neighbourhood $W$ of $\bm b$ such that  $V \times W \subseteq U$ and for each $\bm y \in W$ there is an $\bm x \in V$ with $F(\bm x, \bm y) = \bm c$.
\end{theorem}

Recall that we are looking for a matrix polynomial of degree $k$, with given proper values and a given graph for each non-leading coefficient. The idea is to start with the diagonal matrix Eq.~\eqref{diagonalexample} and perturb the zero off-diagonal entries corresponding to the edges of the graphs to some small nonzero numbers in a symmetric way. As long as the {\em perturbations are sufficiently small}, the implicit function theorem guarantees that the diagonal entries can be adjusted so that the proper values remain unchanged.

Note also that, in the next statement, the assigned graphs $G_0,G_1,\cdots,G_{k-1}$ determine the structure of the coefficients $A_0,\cdots,A_{k-1}$ of $A(z)$.

\begin{theorem}\label{thm:main}
	Let $\lambda_1,\lambda_2,\ldots,\lambda_{nk}$ be $nk$ distinct real numbers, let $\alpha_{k,1},\ldots,\alpha_{k,n}$ be positive (nonzero) real numbers and, for $0 \leq s \leq k-1$, let $G_s$ be a graph on $n$ vertices. 

Then there is an $n \times n$ real symmetric matrix polynomial $A(z) = \sum_{s=0}^{k} A_s z^s$ 
for which:\\
\textup{(a)} the proper values are $\lambda_1,\lambda_2,\ldots,\lambda_{nk}$, \\
\textup{(b)} the leading coefficient is $A_k = \diag[\alpha_{k,1}, \alpha_{k,2}, \ldots, \alpha_{k,n}]$, \\
\textup{(c)} for $s=0,1,\ldots,k-1$, the graph of $A_s$ is $G_s$.
\end{theorem}
\begin{proof}
	Without loss of generality assume that $\lambda_1 < \lambda_2 < \cdots < \lambda_{nk}$. 
Let $G_s$ have $m_s$ edges for $s=0,1,\cdots,k-1$ and  $m = m_0 + \cdots + m_{k-1}$, the total number of edges.
Let $\bm a = (\alpha_{0,1}, \alpha_{0,2}, \ldots, \alpha_{k,n}) \in \mathbb{R}^{nk}$, where $\alpha_{s,r}$ are defined as in Eq.~\eqref{eq:wow}, for $s=0,1,\ldots,k-1$ and $r = 1,2,\ldots,n$, and let $\bm 0$ denote $(0,0,\ldots,0) \in \mathbb{R}^m$. Also, let $A(z)$ be the diagonal matrix polynomial given by Eq.~\eqref{diagonalexample}, which has the given proper values. Recall from Eq.~\eqref{matrixemmata} that $A(z) = M(z,\bm a, \bm 0)$. Let the function $f$ be defined by Eq.~\eqref{functioneff}. Then 
	\begin{equation}
		f\at_{A(z)} = f(z, \bm a, \bm 0) = (\lambda_1,\lambda_2,\ldots,\lambda_{nk}).
	\end{equation}
By Theorem \ref{thm:nonsigularjacobian} the function $f$ has a nonsingular Jacobian at $A(z)$. 

By Theorem \ref{IFT} (the implicit function theorem), there is an open neighbourhood $U \subseteq \mathbb{R}^{nk}$ of $\bm a$ and an open neighbourhood $V \subseteq \mathbb{R}^{m}$ of $\bm 0$ such that for every $\bm \varepsilon \in V$ there is some $\bm{\bar{a}} \in U$ (close to $\bm a$) such that 
	\begin{equation}
	f(z,\bm{\bar{a}},\bm \varepsilon) = (\lambda_1,\lambda_2,\ldots,\lambda_{nk}).
	\end{equation}
	Choose $\bm \varepsilon \in V$ such that none of its entries are zero, and let $\bar{A}(z) = M(z,\bm{\bar{a}},\bm \varepsilon)$. Then $\bar{A}(z)$ has the given proper values, and by definition, the graph of $A_s$ is $G_s$, for $s = 0,1,\ldots,k-1$.
\end{proof}

Note that the proof of Theorem \ref{thm:main} shows only that there is an $m$ dimensional open set of matrices $\bar{A}(z)$ with the given graphs and proper values, and we say nothing about the size of this set. In the quadratic examples of Section \ref{Sec:ex}, the parameter $m$ becomes the total number of springs and dampers. In this context we have:

\begin{corollary}\label{lcor}
	Given graphs $G$ and $H$ on $n$ vertices, a positive definite diagonal matrix $M$, and $2n$ distinct real numbers $\lambda_1, \lambda_2, \ldots, \lambda_{2n}$, there are real symmetric matrices $D$ and $K$ whose graphs are $G$ and $H$, respectively, and the quadratic matrix polynomial $L(z) = M z^2 + D z + K$ has proper values $\lambda_1,\lambda_2, \ldots, \lambda_{2n}$.
\end{corollary}

%%%%%%%%%%%%%%%%%%%%%%%%%%%%%%%%%%%%%%%%%%%%%%%

\section{Numerical Examples}
In this section we provide two numerical examples corresponding to the two systems of Examples \ref{ex:linearsystem} and \ref{ex:systemofmassesandsprings}. Both examples correspond to quadratic systems on four vertices, and in both cases the set of proper values is chosen to be the set of distinct real numbers $\{ -2,-4,\ldots, -16\}$. The existence of matrix polynomials with given proper values and graphs given below is guaranteed by Corollary \ref{lcor}. For a numerical example, we choose all the nonzero off-diagonal entries to be $0.5$. Then the multivariable Newton method is used to approximate the adjusted diagonal entries to arbitrary precision.

We mention in passing that to say ``off-diagonal entries are sufficiently small'' means that Newton's method starts with an initial point sufficiently close to a root. Also, since all the proper values are simple, the iterative method will converge locally. But the detailed analysis of convergence rates and radii of convergence are topics for a separate paper. 

In the following examples we provide an approximation of the coefficient matrices rounded to show ten significant digits. However, the only error in the computations is that of root finding, and in this case, that of Newton's method, and the proper values of the resulting approximate matrix polynomial presented here are accurate to 10 significant digits. The Sage code to carry the computations can be found on github \cite{sagecodeongithub}. 

\begin{example}
Let $\Lambda = \{-2,-4,-6, \ldots, -16\}$, and let the graphs $G$ and $H$ be as shown in Figure \ref{fig:graphsofmatricesdeeandkaytridiagonalagain}. The goal is to construct a quadratic matrix polynomial
\begin{equation} 
	L(z) = M z^2 + D z + K,\hspace*{1cm} M,D,K \in \mathbb{R}^{n\times n},
\end{equation} 
where the graph of $D$ is $H$, the graph of $K$ is $G$ (in this case, both are tridiagonal matrices), and the proper values of $L(z)$ are given by $\Lambda$.
\begin{figure}[h]
\begin{center}
\begin{tikzpicture}[scale=.6,node distance = 1.5cm]
	\begin{scope}[shift={(0,0)},scale=1]
	\node[] () at (-1.5,0) {$G:$};
	\node[draw,circle,fill = light-red,fill opacity=0.2, text opacity = 1] (1) at (0,0) {$1$};
	\node[draw,circle,fill = light-red,fill opacity=0.2, text opacity = 1] (2) [right of = 1] {$2$};
	\node[draw,circle,fill = light-red,fill opacity=0.2, text opacity = 1] (3) [right of = 2] {$3$};
	\node[draw,circle,fill = light-red,fill opacity=0.2, text opacity = 1] (4) [right of = 3] {$4$};
	\draw[very thick,color=dark-blue] (1) -- (2) -- (3) -- (4);
	\end{scope}
	
	\begin{scope}[shift={(0,-4)},scale=1]
	\node[] () at (-1.5,0) {$H:$};

	\node[draw,circle,fill = light-red,fill opacity=0.2, text opacity = 1] (1) at (0,0) {$1$};
	\node[draw,circle,fill = light-red,fill opacity=0.2, text opacity = 1] (2) [right of = 1] {$2$};
	\node[draw,circle,fill = light-red,fill opacity=0.2, text opacity = 1] (3) [right of = 2] {$3$};
	\node[draw,circle,fill = light-red,fill opacity=0.2, text opacity = 1] (4) [right of = 3] {$4$};
	\draw[very thick,color=dark-red] (1) -- (2) -- (3) -- (4);
	\end{scope}
\end{tikzpicture}
\end{center}
\caption{Graphs of $K$ and $D$ of Eq.~\eqref{tridiagonalmatrices}.}
\label{fig:graphsofmatricesdeeandkaytridiagonalagain}
\end{figure}
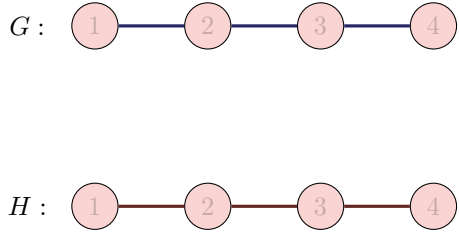 

For simplicity, choose $M$ to be the identity matrix. We start with a diagonal matrix polynomial $A(z)$ whose proper values are given by $\Lambda$:
\begin{small}
\begin{equation}\label{ex:ayofzed}
	A(z) = \left[\begin{array}{rrrr}
1 & 0 & 0 & 0 \\
0 & 1 & 0 & 0 \\
0 & 0 & 1 & 0 \\
0 & 0 & 0 & 1
\end{array}\right] z^2 + \left[\begin{array}{rrrr}
6 & 0 & 0 & 0 \\
0 & 14 & 0 & 0 \\
0 & 0 & 22 & 0 \\
0 & 0 & 0 & 30
\end{array}\right] z + \left[\begin{array}{rrrr}
8 & 0 & 0 & 0 \\
0 & 48 & 0 & 0 \\
0 & 0 & 120 & 0 \\
0 & 0 & 0 & 224
\end{array}\right] 
\end{equation}
\end{small}
Note that the $(1,1)$ entries are the coefficients of $(x-2)(x-4)$, the $(2,2)$ entries are the coefficients of $(x-6)(x-8)$ and so on. Then, perturb all the superdiagonal entries and subdiagonal entries of $A(z)$ to $0.5$ and, using Newton's method, adjust the diagonal entries so that the proper values remain intact. An approximation of the perturbed matrix polynomial $L(z)$ is given by:
\begin{small}
\begin{equation}
D \approx \left[\begin{array}{cccc}
5.86747042533934 & 0.5 & 0 &
0 \\
0.5 & 13.6131619433928 & 0.5 &
0 \\
0 & 0.5 & 21.6432681505587 &
0.5 \\
0 & 0 & 0.5 &
30.8760994807091
\end{array}\right],
\end{equation}
\begin{equation}
K \approx \left[\begin{array}{cccc}
7.74561103829716 & 0.5 & 0 &
0 \\
0.5 & 46.6592230163013 & 0.5 &
0 \\
0 & 0.5 & 119.082534340571 &
0.5 \\
0 & 0 & 0.5 &
240.017612939283
\end{array}\right]
\end{equation}
\end{small}
\qed
\end{example}

\begin{example}
Let $\Lambda = \{-2,-4,-6, \ldots, -16\}$, and let graphs $G$ and $H$ be as shown in Figure \ref{fig:graphsofmatricesdeeandkayagain}. The goal is to construct a quadratic matrix polynomial
\begin{equation} 
	L(z) = M z^2 + D z + K,\hspace*{1cm} M,D,K \in \mathbb{R}^{n\times n},
\end{equation} 
where the graph of $D$ is $H$, the graph of $K$ is $G$, and the proper values of $L(z)$ are given by $\Lambda$.
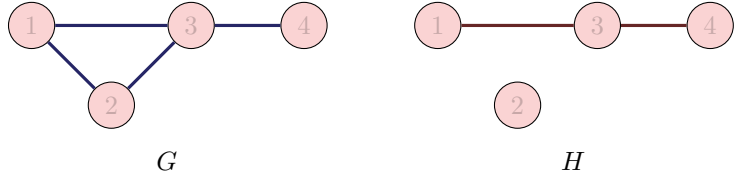
\begin{figure}[h]
\begin{center}
\begin{tikzpicture}[scale=.6,node distance = 1.5cm]
	\begin{scope}[shift={(0,0)},scale=1]
	\node[] () at (3,-3) {$G$};
	\node[draw,circle,fill = light-red,fill opacity=0.2, text opacity = 1] (1) at (0,0) {$1$};
	\node[draw,circle,fill = light-red,fill opacity=0.2, text opacity = 1] (2) [below right of = 1] {$2$};
	\node[draw,circle,fill = light-red,fill opacity=0.2, text opacity = 1] (3) [above right of = 2] {$3$};
	\node[draw,circle,fill = light-red,fill opacity=0.2, text opacity = 1] (4) [right of = 3] {$4$};
	\draw[very thick,color=dark-blue] (4) -- (3) -- (2) -- (1) -- (3);
	\end{scope}
	
	\begin{scope}[shift={(9,0)},scale=1]
	\node[] () at (3,-3) {$H$};
	\node[draw,circle,fill = light-red,fill opacity=0.2, text opacity = 1] (1) at (0,0) {$1$};
	\node[draw,circle,fill = light-red,fill opacity=0.2, text opacity = 1] (2) [below right of = 1] {$2$};
	\node[draw,circle,fill = light-red,fill opacity=0.2, text opacity = 1] (3) [above right of = 2] {$3$};
	\node[draw,circle,fill = light-red,fill opacity=0.2, text opacity = 1] (4) [right of = 3] {$4$};
	\draw[very thick,color=dark-red] (4) -- (3) -- (1);
	\end{scope}
\end{tikzpicture}
\end{center}
\caption{Graphs of $K$ and $D$.}
\label{fig:graphsofmatricesdeeandkayagain}
\end{figure}  

Choose $M$ to be the identity matrix and start with the same diagonal matrix polynomial $A(z)$ as in Eq.~\eqref{ex:ayofzed}. Perturb those entries of $A(z)$ corresponding to an edge to $0.5$ and, using  Newton's method, adjust the diagonal entries so that the proper values are not perturbed. An approximation of the matrix polynomial $L(z)$ is given by:
\begin{small}
\begin{equation}
D \approx \left[\begin{array}{cccc}
5.96497947933414 & 0 & 0.5 &
0 \\
0 & 13.9962664239873 & 0 &
0 \\
0.5 & 0 & 21.2163179014646 &
0.5 \\
0 & 0 & 0.5 &
30.8224361952140
\end{array}\right],
\end{equation}
\begin{equation}
 K \approx \left[\begin{array}{cccc}
7.94384133116825 & 0.5 & 0.5 &
0 \\
0.5 & 48.0284454626440 & 0.5 &
0 \\
0.5 & 0.5 & 113.276104063793 &
0.5 \\
0 & 0 & 0.5 &
239.067195294473
\end{array}\right].
\end{equation}
\end{small}
\qed
\end{example}

%%%%%%%%%%%%%%%%%%%%%%%%%%%%%%%%%%%%%%%%%%%%%%
\section{Conclusions}

Linked vibrating systems consisting of a collection of rigid components connected by springs and dampers require the spectral analysis of matrix functions of the form Eq.~\eqref{eq1}. As we have seen, mathematical models for the analysis of such systems have been developed by
Chu and Golub (\cite{cdy07,cg02,cg05}) and by Gladwell \cite{g04}, among others. The mass distribution in these models is just that of the components, and elastic and dissipative
properties are associated with the {\em linkage} of the parts, rather than the parts themselves.

Thus, for these models, the leading coefficient (the mass matrix) is a positive definite diagonal matrix. The damping and stiffness matrices have a zero-nonzero structure dependent on graphs (e.g. tridiagonal for a path) which, in turn, determine the {\em connectivity} of the components of the system.

In this paper a technique has been developed for the solution of some {\em inverse} vibration problems
in this context for matrix polynomials of a general degree $k$ as in Eq.~\eqref{eqa}, and then the results are applied to the specific case of quadratic polynomials, with significant applications. Thus, given a real spectrum for the system, we show how corresponding real coefficient matrices $M$, $D$, and $K$ can be found, and numerical examples are included. The technique applies equally well to some higher-order differential systems, and so the theory has been 
developed in that context.

In principle, the method developed here could be extended to the designs of systems with some (possibly all) non-real proper values appearing in conjugate pairs as is done for the linear case in \cite{h17}.

\section*{Acknowledgement}
Funding: This work is supported by the Natural Sciences and Engineering Research Council of Canada (NSERC).

\section*{References}
\bibliography{ref}

\begin{thebibliography}{10}
\expandafter\ifx\csname url\endcsname\relax
  \def\url#1{\texttt{#1}}\fi
\expandafter\ifx\csname urlprefix\endcsname\relax\def\urlprefix{URL }\fi
\expandafter\ifx\csname href\endcsname\relax
  \def\href#1#2{#2} \def\path#1{#1}\fi

\bibitem{g04}
G.~M.~L. Gladwell, Inverse Problems in Vibration, 2nd ed., Martinus Nijhoff,
  Dordrecht, Netherlands, 2004.

\bibitem{cg02}
M.~T. Chu, G.~H. Golub, Structured inverse eigenvalue problems, Acta Numerica
  (2002) 1--71\href {http://dx.doi.org/10.1017/S0962492902000014}
  {\path{doi:10.1017/S0962492902000014}}.

\bibitem{cg05}
M.~T. Chu, G.~H. Golub, Inverse Eigenvalue Problems: Theory, Algorithms, and
  Applications, Numerical Mathematics and Scientific Computation, Oxford
  University Press, 2005.

\bibitem{h17}
K.~{Hassani Monfared}, \href{http://www.sciencedirect.com/science/article/pii/
  S002437951730229X}{Existence of a not necessarily symmetric matrix with given
  distinct eigenvalues and graph}, Linear Algebra and its Applications 527
  (2017) 1--11.
\newblock \href {http://dx.doi.org/10.1016/j.laa.2017.04.006}
  {\path{doi:10.1016/j.laa.2017.04.006}}.
\newline\urlprefix\url{http://www.sciencedirect.com/science/article/pii/
  S002437951730229X}

\bibitem{tm01}
F.~Tisseur, K.~Meerbergen, The quadratic eigenvalue problem, SIAM Review 43
  (2001) 235--286, see also http://www.ma.man.ac.uk/\verb+~+ftisseur.

\bibitem{nk01}
N.~K. Nichols, J.~Kautsky, Robust eigenstructure assignment in quadratic matrix
  polynomials: nonsingular case, SIAM Journal on Matrix Analysis and
  Applications 23~(1) (2001) 77--102 (electronic).

\bibitem{der97}
B.~N. Datta, S.~Elhay, Y.~M. Ram, Orthogonality and partial pole assignment for
  the symmetric definite quadratic pencil, Linear Algebrea and its Applications
  257 (1997) 29--48.

\bibitem{cd96}
E.~K.-W. Chu, B.~Datta, Numerical robust pole assignment for second-order
  systems, International Journal of Control 64 (1996) 1113--1127.

\bibitem{c02}
E.~K. Chu, Pole assignment for second-order systems, Mechanical Systems and
  Signal Processing 16 (2002) 39--59.

\bibitem{l07}
P.~Lancaster, Inverse spectral problems for semisimple damped vibrating
  systems, SIAM Journal on Matrix Analysis and Applications 29 (2007) 279--301.

\bibitem{bcs07}
Z.-J. Bai, D.~Chu, D.~Sun, \href{https://doi.org/10.1137/060656346}{A dual
  optimization approach to inverse quadratic eigenvalue problems with partial
  eigenstructure}, SIAM Journal on Scientific Computing 29~(6) (2007)
  2531--2561.
\newblock \href {http://dx.doi.org/10.1137/060656346}
  {\path{doi:10.1137/060656346}}.
\newline\urlprefix\url{https://doi.org/10.1137/060656346}

\bibitem{cklx09}
Y.-F. Cai, Y.-C. Kuo, W.-W. Lin, , S.-F. Xu, Solutions to a quadratic inverse
  eigenvalue problem, Linear Algebra and its Applications 430 (2009)
  1590--1606.

\bibitem{yd11}
Y.~Yuan, H.~Dai, \href{http://www.sciencedirect.com/science/article/pii/
  S0024379510003265}{Solutions to an inverse monic quadratic eigenvalue
  problem}, Linear Algebra and its Applications 434~(11) (2011) 2367--2381,
  special Issue: Devoted to the 2nd NASC 08 Conference in Nanjing (NSC).
\newblock \href {http://dx.doi.org/http://dx.doi.org/10.1016/j.laa.2010.06.030}
  {\path{doi:http://dx.doi.org/10.1016/j.laa.2010.06.030}}.
\newline\urlprefix\url{http://www.sciencedirect.com/science/article/pii/
  S0024379510003265}

\bibitem{re96}
Y.~M. Ram, S.~Elhay, An inverse eigenvalue problem for the symmetric
  tridiagonal quadratic pencil with application to damped oscillatory systems,
  SIAM Journal on Applied Mathematics 56~(1) (1996) 232--244.

\bibitem{b08}
Z.-J. Bai, \href{http://stacks.iop.org/0266-5611/24/i=1/a=015005}{Symmetric
  tridiagonal inverse quadratic eigenvalue problems with partial eigendata},
  Inverse Problems 24 (2008) 015005.
\newblock \href
  {http://dx.doi.org/https://doi.org/10.1088/0266-5611/24/1/015005}
  {\path{doi:https://doi.org/10.1088/0266-5611/24/1/015005}}.
\newline\urlprefix\url{http://stacks.iop.org/0266-5611/24/i=1/a=015005}

\bibitem{lz14}
P.~Lancaster, I.~Zaballa, \href{http://dx.doi.org/10.1137/130905216}{{On the
  Inverse Symmetric Quadratic Eigenvalue Problem}}, SIAM Journal on Matrix
  Analysis and Applications 35 (2014) 254--278.
\newblock \href {http://dx.doi.org/10.1137/130905216}
  {\path{doi:10.1137/130905216}}.
\newline\urlprefix\url{http://dx.doi.org/10.1137/130905216}

\bibitem{cdy07}
M.~T. Chu, N.~D. Buono, B.~Yu, Structured quadratic inverse eigenvalue problem,
  {I}. serially linked systems, SIAM Journal on Scientific Computing 29 (2007)
  2668--2685.
\newblock \href {http://dx.doi.org/10.1137/060672510}
  {\path{doi:10.1137/060672510}}.

\bibitem{dlc09}
B.~Dong, M.~M. Lin, M.~T. Chu,
  \href{http://www.sciencedirect.com/science/article/pii/
  S0022460X09005513}{Parameter reconstruction of vibration systems from partial
  eigeninformation}, Journal of Sound and Vibration 327~(3--5) (2009) 391--401.
\newblock \href {http://dx.doi.org/10.1016/j.jsv.2009.06.026}
  {\path{doi:10.1016/j.jsv.2009.06.026}}.
\newline\urlprefix\url{http://www.sciencedirect.com/science/article/pii/
  S0022460X09005513}

\bibitem{ldc10-1}
M.~M. Lin, B.~Dong, M.~T. Chu,
  \href{http://stacks.iop.org/0266-5611/26/i=6/a=065003}{Inverse mode problems
  for real and symmetric quadratic models}, Inverse Problems 26~(6) (2010)
  065003.
\newblock \href {http://dx.doi.org/10.1088/0266-5611/26/6/065003}
  {\path{doi:10.1088/0266-5611/26/6/065003}}.
\newline\urlprefix\url{http://stacks.iop.org/0266-5611/26/i=6/a=065003}

\bibitem{ldc10-2}
M.~M. Lin, B.~Dong, M.~T. Chu, Semi-definite programming techniques for
  structured quadratic inverse eigenvalue problems, Numerical Algorithms 53
  (2010) 419--437.
\newblock \href {http://dx.doi.org/10.1007/s11075-009-9309-9}
  {\path{doi:10.1007/s11075-009-9309-9}}.

\bibitem{hs13}
K.~{Hassani Monfared}, B.~L. Shader,
  \href{http://www.sciencedirect.com/science/article/pii/
  S0024379513001006}{Construction of matrices with a given graph and prescribed
  interlaced spectral data}, Linear Algebra and its Applications 438 (2013)
  4348--4358.
\newblock \href {http://dx.doi.org/10.1016/j.laa.2013.01.036}
  {\path{doi:10.1016/j.laa.2013.01.036}}.
\newline\urlprefix\url{http://www.sciencedirect.com/science/article/pii/
  S0024379513001006}

\bibitem{bm08}
J.~Bondy, U.~Murty, Graph Theory, Springer, 2008.

\bibitem{lmz03}
P.~Lancaster, A.~S. Markus, F.~Zhou,
  \href{http://dx.doi.org/10.1137/S0895479803423792}{Perturbation theory for
  analytic matrix functions: The semisimple case}, SIAM Journal on Matrix
  Analysis and Applications 25~(3) (2003) 606--626.
\newblock \href {http://dx.doi.org/10.1137/S0895479803423792}
  {\path{doi:10.1137/S0895479803423792}}.
\newline\urlprefix\url{http://dx.doi.org/10.1137/S0895479803423792}

\bibitem{dr14}
A.~L. Dontchev, R.~T. Rockafeller, Implicit Functions and Solution Mappings: A
  View from Variational Analysis, 2nd ed., Springer Series in Operations
  Research and Financial Engineering, Springer, 2014.

\bibitem{kran02}
S.~G. Krantz, H.~R. Parks, The Implicit Function Theorem: History, Theory, and
  Applications, Birkh\"{a}user, Boston, 2002.

\bibitem{sagecodeongithub}
K.~{Hassani Monfared}, Sage code on github,
  \url{https://github.com/k1monfared/lambda_ispmpg}.

\end{thebibliography}

\end{document}